\documentclass[11pt,reqno]{amsart} 

\usepackage[utf8]{inputenc} 
\usepackage{graphicx}
\usepackage{epstopdf}
 \usepackage{amsaddr}

\usepackage[colorlinks=true]{hyperref}
\newtheorem{theorem}{Theorem}[section]
\newtheorem{corollary}{Corollary}
\newtheorem{main}{Main Theorem}
\newtheorem{lemma}[theorem]{Lemma}
\newtheorem{proposition}{Proposition}

\newtheorem{example}{Example}

\theoremstyle{definition}
\newtheorem{definition}[theorem]{Definition}
\newtheorem{remark}{Remark}

\title[Application of HHD to the study of certain vector fields]{Application of Helmholtz--Hodge decomposition to the study of certain vector fields}
\author{Tomoharu Suda}

\begin{document}

\begin{abstract}
 Smooth vector fields on $\mathbb{R}^n$ can be decomposed into the sum of a gradient vector field and divergence-free (solenoidal) vector field under suitable hypotheses. This is called the Helmholtz--Hodge decomposition (HHD), which has been applied to analyze the topological features of vector fields. In this study, we apply the HHD to study certain types of vector fields. In particular, we investigate the existence of strictly orthogonal HHDs, which assure an effective analysis. The first object of the study is linear vector fields. We demonstrate that a strictly orthogonal HHD for a vector field of the form $\textbf{F}(\textbf{x}) = A \textbf{x}$ can be obtained by solving an algebraic Riccati equation. Subsequently, a method to explicitly construct a Lyapunov function is established. In particular, if A is normal, there exists an easy solution to this equation. Next, we study planar vector fields. In this case, the HHD yields a complex potential, which is a generalization of the notion in hydrodynamics with the same name. We demonstrate the convenience of the complex potential formalism by analyzing vector fields given by homogeneous quadratic polynomials.
\end{abstract}
\maketitle
\section{Introduction}
Helmholtz-Hodge decomposition is a decomposition of vector fields, where they are represented as the sum of a gradient vector field and divergence-free (solenoidal) vector field. (The exact definition of the HHD is provided later in Section \ref{pre}.) It is a classical method widely used in applications. Many will first encounter the HHD in the study of electromagnetism, where it is applied to the construction of scalar and the vector potentials. Additionally, the HHD has been applied in hydrodynamics \cite{Denaro_On_2003}.

The application of the HHD is not limited to the study of physics. Recently, several methods have been proposed to use it as a tool to extract the topological feature of vector fields. These methods depend on the numerical construction of the HHD, and they are primarily intended as visualization methods. For a review, the reader is referred to Bhatia, Norgard, and Pascucci \cite{Bhatia_The_2013}.

However, the HHD may be computed rigorously if the vector fields are expressed in an explicit form. This will enable us to rigorously analyze the behavior of vector fields. Examples of results suitable for this purpose are scarce. Demongeot, Grade and Forest studied planar cases \cite{Demongeot_Li_2007, Demongeot_Li_2007_2}. Mendes and Duarte considered the approximation of general vector fields by divergence-free vector fields \cite{Duarte_Deformation_1983}. Chukanov and Ulyanov applied a special case of HHD for the stability analysis of linear systems in the context of control systems \cite{simple_dec}. Additionally, the author has proposed a method to construct Lyapunov functions using the HHD \cite{suda2019construction}. 

Although these methods are applicable, theoretically, to a relatively wide class of vector fields, their efficacy is not always guaranteed. The efficacy of an analysis depends on the setting of boundary conditions. If the boundary conditions do not match the dynamics of the system studied, artifacts may appear and a practical result cannot be obtained. An exception might be the natural HHD, which is ``intrinsic'' and  independent of boundary conditions \cite{Bhatia_The_2014}. However, it is not suitable for the mathematical analysis of vector fields in general because it requires a definition for ``inside'' or ``outside'' of the domain of interest. 

Therefore, an intrinsic and exact condition is required to assure that HHD-based methods are applicable. In an earlier paper by the author, it was suggested that the existence of a strictly orthogonal HHD ensures the efficacy of an analysis based on the HHD \cite{suda2019construction}. This is the case where the gradient vector field and the divergence-free vector fields are orthogonal everywhere. If this condition holds, we can ``read off'' information regarding the dynamics by studying each component of the decomposition. As such,  vector fields with a strictly orthogonal HHD may be regarded as a generalization of gradient vector fields. In particular, a Lyapunov function is easily obtained. However, they are difficult to construct, and their existence is not obvious.

In this study, we investigate the existence of a strictly orthogonal HHD for certain types of vector fields using constructive methods. Furthermore, we illustrate its use and limitations. Although the vector fields studied herein are linear or planar and far from being general, the analysis of their properties will clarify the basic principle. Information regarding their behavior is essential in constructing a theory applicable to a wider class of systems.

Regarding the construction of the Lyapunov function, SDE decomposition, which was introduced by Ao, is closely related to the HHD-based method considered here \cite{ao2004potential, kwon2005structure, yuan2017sde}. We also discuss the relationship between the former and the latter.

Now, we present the main results of this study.
First, we consider HHD for a linear system expressed by an ordinary equation of the following form: 
\[\frac{d {\bf x}}{d t} = A {\bf x},\]
where ${\bf x}$ is the unknown $\mathbb{R}^n$-valued function and $A$ is an $n\times n$ matrix. The following characterization enables us to calculate a strictly orthogonal HHD.
\begin{main}\label{thm_ch}
Let ${\bf F}({\bf x})= A {\bf x}$ be a vector field on $\mathbb{R}^n$, where $A$ is an  $n\times n$ rectangular matrix. Then, every strictly orthogonal HHD for ${\bf F}$ is in the form $A {\bf x} = -P {\bf x} + H{\bf x}, $
where $P$ is a symmetric matrix with ${\rm tr}(P) = -{\rm tr}(A)$ and satisfies
\begin{equation}\label{lyapunov_eqn}
2 P^2 + A^T P + PA = O.
\end{equation}
$H$ is found by taking $H = A + P.$ Conversely, if a symmetric matrix $P$ with ${\rm tr}(P) = -{\rm tr}(A)$ satisfies the equation (\ref{lyapunov_eqn}), then $A = -P + H,$ where $H = A + P,$ gives a strictly orthogonal HHD.
\end{main}
In particular, if the coefficient matrix is normal, we can easily find a strictly orthogonal HHD.
\begin{main}\label{thm_norm}
Let ${\bf F}({\bf x})= A {\bf x}$ be a vector field on $\mathbb{R}^n$, where $A$ is a normal $n\times n$ rectangular matrix.  Then there is a strictly orthogonal HHD of ${\bf F}$.
\end{main}
Next, we consider the planar vector fields. In this case, the HHD is of a simpler form and we may further formulate it in terms of a complex potential, which is a generalization of the concept with the same name in hydrodynamics. As in hydrodynamics, complex potentials render a significantly simpler analysis. With this formulation, the problem of constructing a strictly orthogonal HHD is rephrased into a problem of solving a differential equation. Using this result, we can demonstrate the next theorem. 
\begin{main}\label{thm_quad}
Let there be given a system of ordinary equations of the following form:
\[
\begin{split}
\frac{d x}{d t} &= f(x,y) = p_1 x^2 + q_1 xy + r_1 y^2,\\
\frac{d y}{d t} &= g(x,y) = p_2 x^2 + q_2 xy + r_2 y^2,
\end{split}
\]
where $p_i, q_i, r_i$ $(i=1,2)$ are real constants. If the coefficients satisfy the condition
\begin{equation}\label{cond_ohhd_e}
q_1^2 -2 (p_2-r_2) q_1 -4 p_2 r_2 + q_2^2 +2 (p_1-r_1) q_2 -4 p_1 r_1 = 0,
\end{equation}
then there exists a strictly orthogonal HHD of the vector field ${\bf F}(x,y) = \left(f(x,y),g(x,y)\right)^T$.
\end{main}

The remainder of this paper is organized as follows. In Section \ref{pre}, we present the basic terminologies and a few preliminary results from the literature. In Section \ref{a_li_sys}, we describe the application of the HHD to linear systems, which results in a method to extract information on the behavior of orbits. In Section \ref{a_pla_sys}, we discuss the planar systems. In this case, the HHD assumes a form that is analogous to the complex potential in hydrodynamics. Finally, in Section \ref{con_rem}, we present the concluding remarks.

\section{Preliminaries}\label{pre}
In this section, we present the basic terminologies and a few preliminary results from the literature.
First, we present the definition of the Helmholtz--Hodge decomposition as follows.
\begin{definition}[Helmholtz--Hodge decomposition]
\normalfont
For a vector field $ {\bf F}$ on a domain $\Omega \subset {\mathbb R}^n$, its {\bf Helmholtz--Hodge decomposition} ({\bf HHD}) is a decomposition of the form
$$ {\bf F} = -\nabla V + {\bf u},$$
where $V: \Omega \rightarrow\mathbb{R}$ is a $C^3$ function and ${\bf u}$ is a vector field on $\Omega$ with the property $\nabla \cdot {\bf u} = 0$. $V$ is called a {\bf potential function}.
\end{definition}

The existence of an HHD can be assured if the next Poisson equation has a solution on the domain of interest.
$$\Delta V = - \nabla \cdot {\bf F}.$$

In general, the HHD is not unique. Furthermore, it does not necessarily reflect the behavior of the original vector field. However, the decomposition of the next type preserves much information in this aspect, and it is a useful tool for the analysis if available.

\begin{definition}[S. \cite{suda2019construction}]
The HHD ${\bf F} =- \nabla V+ {\bf u} $ is said to be {\bf strictly orthogonal} on $D \subset {\mathbb R}^n$ if ${\bf u}({\bf x}) \cdot \nabla V({\bf x})= 0$ for all ${\bf x}\in D$.
\end{definition}
If ${\bf F} =- \nabla V+ {\bf u} $ is a strictly orthogonal HHD, then $V$ does not increase along the solutions. This is observed by
$$ \dot V({\bf x}) := \nabla V ({\bf x})\cdot {\bf F}({\bf x}) = - |\nabla V({\bf x})|^2 \leq0.$$
Here the function $\dot V := \nabla V ({\bf x})\cdot {\bf F}({\bf x}) $ is called the {\bf orbital derivative} of $V$. The condition $\dot V({\bf x}) \leq 0$ ensures that the orbits of the system ``slips down'' the (hyper-)surface given by $z = V({\bf x})$. Therefore, we may study the behavior of orbits by considering the level sets of $V.$ In particular, if an equilibrium point is stable, $V$ becomes a Lyapunov function. 
\begin{theorem}[S. \cite{suda2019construction}]\label{ohhd}
Let $D \subset {\mathbb R}^n$ be a bounded domain. If a vector field ${\bf F}:{\mathbb R}^n \to {\mathbb R}^n$ has a strictly orthogonal HHD on $\bar D$ with potential function $V$, then the following hold:
\begin{enumerate}
\item If $D$ is forward invariant, then for all ${\bf x} \in D$, $$\omega({\bf x}) \subset \{{\bf y} \in {\bar D}\,| \,\nabla V({\bf y}) = {\bf 0}\}.$$
\item If $D$ is backward invariant, then for all ${\bf x} \in D$,$$\alpha({\bf x}) \subset \{{\bf y} \in {\bar D}\,| \,\nabla V({\bf y}) = {\bf 0}\}.$$
\end{enumerate}
\end{theorem}
These observations form the base of the analysis method proposed herein. 
 Details of the use of the strictly orthogonal HHD are provided in \cite{suda2019construction}.

The next estimate will be used later. Though it is elementary, we include its proof for completeness.
\begin{lemma}\label{lem_p_g}
Let $h: \mathbb{R}^n \to \mathbb{R}$ be a harmonic function, satisfying the following inequality for all ${\bf x} \in \mathbb{R}^n$:
$$|h({\bf x}) - h({\bf 0})| \leq C || {\bf x} ||^k,$$
where $C>0$ is a constant and $k$ is a positive integer. Then, $h$ is a polynomial of degree at the most $k.$
\end{lemma}
\begin{proof}
Without loss of generality, we may assume that $h({\bf 0}) = 0.$

Let us fix ${\bf x} \in \mathbb{R}^n$ and demonstrate that $D^\alpha h ({\bf x}) = 0$ if $\alpha$ is a multi-index of length $l > k.$

Let us recall Cauchy's estimates (Theorem 2.4 in \cite{axler2001harmonic}): If a harmonic function $u$ is bounded by $M >0 $ on $B_{r}({\bf x}) = \{{\bf y} \mid ||{\bf y} - {\bf x} || = r\},$ then we have the next inequality.
$$|D^\alpha u ({\bf x})|\leq \frac{C_{\alpha}}{r^l} M,$$
where $C_\alpha$ is a constant depending on the multi-index $\alpha$ and $l$ is the length of $\alpha.$

From the hypothesis, we have
$$|h({\bf y})| \leq C (r + ||{\bf x}||)^k,$$
for all ${\bf y} \in B_{r}({\bf x}).$ Therefore, we can apply Cauchy's estimates with $M = C (r + ||{\bf x}||)^k,$ and obtain
$$|D^\alpha h ({\bf x})|\leq \frac{C_\alpha C}{r^{l-k}} \left(1 + \frac{||{\bf x}||}{r} \right)^k.$$

By considering the limit as $r \to \infty,$ we have $D^\alpha h ({\bf x}) = 0$ if $l > k.$ Because $h$ is analytic, we conclude that $h$ is a polynomial of order at most $k.$
\end{proof}

\section{Analysis of linear systems}\label{a_li_sys}
In this section, we consider the HHD for linear systems, that is, vector fields are given in the form ${\bf F}({\bf x})= A {\bf x},$ where $A$ is a rectangular matrix. The purpose is to construct a strictly orthogonal HHD for linear vector fields and thereby obtain a method to construct Lyapunov functions.

Once we have a method to construct Lyapunov functions for linear systems, equilibrium points of nonlinear systems can be analyzed by the linear approximation.

 It should be noted that the Lyapunov functions obtained here can be constructed also by the method of SDE decomposition, as we discuss later.   

In what follows, we consider the HHD defined on $\mathbb{R}^n.$ Therefore, we say that an HHD exists if it is valid on $\mathbb{R}^n.$

First, we note that HHD for linear vector fields is characterized by a decomposition of the coefficient matrix.
\begin{lemma}\label{HHD_mat}
Let ${\bf F}({\bf x})= A {\bf x}$ be a linear vector field on $\mathbb{R}^n$, where $A \in {\rm Mat}(n, \mathbb{R})$. Then, there exists an HHD of ${\bf F}({\bf x})$ if and only if there exists a symmetric $P\in {\rm Mat}(n, \mathbb{R})$ and a matrix $H \in {\rm Mat}(n, \mathbb{R})$ with the following properties:
\begin{eqnarray*}
A &=& - P + H,\\
{\rm tr}(A) &=& - {\rm tr}(P),\\
{\rm tr}(H) &=& 0.
\end{eqnarray*}
\end{lemma}
\begin{proof}
Let there exist an HHD of ${\bf F}({\bf x})$, ${\bf F} =- \nabla V+ {\bf u}.$ By considering the Jacobian matrix at ${\bf 0},$ we obtain
$$ A = \textbf{J}_{\textbf{F}} ({\bf 0}) = -\textbf{J}_{\nabla V}({\bf 0}) + \textbf{J}_{\textbf{u}}({\bf 0}),$$
where $\textbf{J}_{\textbf{F}} ({\bf 0})$ denotes the Jacobian matrix of $\textbf{F}$ at $\textbf{0}.$
In general, we have ${\rm tr}\left(\textbf{J}_{\textbf{v}} ({\bf 0})\right) = \left(\nabla \cdot {\bf v}\right)({\bf 0})$ for a vector field ${\bf v}$ on $\mathbb{R}^n.$ Furthermore, it is easily observed that $\textbf{J}_{\nabla V}({\bf 0})$ is the Hessian matrix of $V$ at $\textbf{0}$. Therefore, $P = \textbf{J}_{\nabla V}({\bf 0})$ and $H=\textbf{J}_{\textbf{u}}({\bf 0})$ satisfy the conditions stated above.

For the converse, it suffices if we set
\[\begin{split}
V({\bf x}) &= \frac{1}{2}{\bf x}^{T} P {\bf x},\\
{\bf u}({\bf x}) &= H {\bf x}.
\end{split}\] \end{proof}

For any $A \in {\rm Mat}(n, \mathbb{R})$, it is easy to observe that the conditions stated in Lemma \ref{HHD_mat} are satisfied if we take
\[
\begin{aligned}
	 P &= -\frac{1}{2} (A + A^T),\\
	H &= \frac{1}{2} (A - A^T).
\end{aligned}
\]
Thus we obtain the following result, which guarantees the existence of HHD for linear vector fields.
\begin{corollary}
Let ${\bf F}({\bf x})= A {\bf x}$ be a linear vector field on $\mathbb{R}^n$, where $A \in {\rm Mat}(n, \mathbb{R})$. Then, there exists an HHD of ${\bf F}({\bf x}).$
\end{corollary}
The existence of a strictly orthogonal HHD is characterized as follows.
\begin{corollary}\label{cor_so}
Let ${\bf F}({\bf x})= A {\bf x}$ be a linear vector field on $\mathbb{R}^n$, where $A \in {\rm Mat}(n, \mathbb{R})$. Then there exists a strictly orthogonal HHD of ${\bf F}({\bf x})$ if and only if there exist a symmetric $P\in {\rm Mat}(n, \mathbb{R})$ and a matrix $H \in {\rm Mat}(n, \mathbb{R})$ with
$$ PH + H^T P= O,$$
in addition to the properties stated in Lemma \ref{HHD_mat}.
\end{corollary}
\begin{proof}
If a strictly orthogonal HHD of ${\bf F}({\bf x})$ exists, then we have ${\bf u}({\bf x}) \cdot \nabla V({\bf x})= 0$ for all ${\bf x}.$ By considering the derivatives, we obtain $ PH + H^T P= O.$

The converse is evident.
\end{proof}
From the proof of Lemma \ref{HHD_mat}, if a pair of matrices $(P,H)$ satisfies the conditions that $A=-P+H$, $P$ is symmetric and ${\rm tr}(A) = - {\rm tr}(P)$, then we have an HHD of the vector field $\textbf{F} = A \textbf{x}$ given by $A \textbf{x} = -P\textbf{x} + H \textbf{x}.$
This observation motivates the following definition.
\begin{definition}
\normalfont
Let $A \in {\rm Mat}(n, \mathbb{R})$ be a matrix. {\bf Helmholtz--Hodge decomposition (HHD)} of $A$ is a decomposition of the form $A = -P +H,$ where $P$ is symmetric, ${\rm tr}(A) = - {\rm tr}(P)$, and ${\rm tr}(H) = 0.$

Helmholtz-Hodge decomposition $A = -P +H$ of $A$ is said to be {\bf strictly orthogonal} if $ PH + H^T P= O$ holds.
\end{definition}
Because the HHD is obtained by adding two vector fields of different variances, it is not preserved under a change of coordinates. However, the HHD of the matrices is preserved, in a sense, under orthogonal transformations.
\begin{lemma}\label{lem_pres}
Let $A \in {\rm Mat}(n, \mathbb{R})$ be a matrix with an HHD. Let $S\in {\rm Mat}(n, \mathbb{R})$ be an orthogonal matrix. Then, there exists an HHD of $B =S A S^{T}$. If the decomposition of $A$ is given by $A = -P +H$, then $B$ is decomposed by $Q := S P S^T$ and $G:= SHS^T$ into the form $B = -Q +G.$
\end{lemma}
\begin{proof}
Direct calculation.
\end{proof}
\begin{remark}
\normalfont
It is noteworthy that the change of coordinates ${\bf y} = S {\bf x}$ does not yield the decomposition stated in Lemma \ref{lem_pres}. Lemma \ref{lem_pres} is better interpreted as a construction of a decomposition {\it ab initio}.
\end{remark}
\begin{corollary}\label{cor_pres}
Let $A \in {\rm Mat}(n, \mathbb{R})$ be a matrix with a strictly orthogonal HHD. Let $S\in {\rm Mat}(n, \mathbb{R})$ be an orthogonal matrix. Then, there exists a strictly orthogonal HHD of $B =S A S^{T}$. If a decomposition of $A$ is given by $A = -P +H$, $B$ is decomposed by $Q := S P S^T$ and $G:= SHS^T.$
\end{corollary}

In general, we may use the following result to calculate strictly orthogonal HHD.

\begin{theorem}[Theorem \ref{thm_ch}]
Let ${\bf F}({\bf x})= A {\bf x}$ be a vector field on $\mathbb{R}^n$, where $A$ is an  $n\times n$ rectangular matrix. Then, every strictly orthogonal HHD for ${\bf F}$ is of the form $A {\bf x} = -P {\bf x} + H{\bf x}, $
where $P$ is a symmetric matrix with ${\rm tr}(P) = -{\rm tr}(A)$ and satisfies
\begin{equation}\label{lya_eqn}
2 P^2 + A^T P + PA = O.
\end{equation}
$H$ is found by taking $H = A + P.$ Conversely, if a symmetric $n\times n$ matrix $P$ with ${\rm tr}(P) = -{\rm tr}(A)$ satisfies equation (\ref{lya_eqn}), then $A = -P + H,$ where $H = A + P,$ gives a strictly orthogonal HHD.
\end{theorem}

\begin{proof}
Let a strictly orthogonal HHD for ${\bf F}$ be given by
\begin{equation}\label{eqn_HHD_pf}
{\bf F} =- \nabla V+ {\bf u}.
\end{equation}
By Lemma \ref{HHD_mat}, we have the HHD of matrix $A = -P + H$, where $P = \textbf{J}_{\nabla V}({\bf 0})$, $H=\textbf{J}_{\textbf{u}}({\bf 0})$ and ${\rm tr}(P) = -{\rm tr}(A)$. By Corollary \ref{cor_so}, $P$ and $H$ satisfy $ PH + H^T P= O.$

First we demonstrate that $\nabla V({\bf x}) = P {\bf x}$ and ${\bf u}({\bf x}) = H {\bf x}.$ Because $\nabla V({\bf 0})$ and ${\bf u}({\bf 0}) $ are orthogonal and $\textbf{F}(\textbf{0}) = \textbf{0}$, we have $\nabla V({\bf 0}) = {\bf u}({\bf 0}) = {\bf 0}.$ By subtracting $A {\bf x} = -P {\bf x} + H {\bf x}$ from equation (\ref{eqn_HHD_pf}), we have
\begin{equation}\label{harmonic_prop}
\begin{split}
0 &= -\left(\nabla V({\bf x}) - P{\bf x}\right) + \left( {\bf u}({\bf x}) - H {\bf x}\right)\\
 &= -\nabla \tilde V({\bf x}) + \tilde{\bf u}({\bf x}),
\end{split}
\end{equation}
where
\begin{eqnarray*}
\tilde V({\bf x}) &=& V({\bf x}) -\frac{1}{2} {\bf x}^T P{\bf x}\\
\tilde{\bf u}({\bf x}) &=& {\bf u}({\bf x}) - H {\bf x}.
\end{eqnarray*}
Because we have $\nabla \cdot \tilde{\bf u}({\bf x}) = 0$ and identity (\ref{harmonic_prop}),  $\tilde V$ is harmonic. Furthermore, by the definition of $\tilde V,$ the gradient and the Hessian matrix of $\tilde V$ vanish at $\textbf{0}.$

Next, we demonstrate that an estimate of the following form holds for all ${\bf x} \in \mathbb{R}^n$:
\begin{equation}
|\tilde V ({\bf x})- \tilde V({\bf 0})| \leq C ||{\bf x}||^2,
\end{equation}
where $C$ is a positive constant.

Indeed, using the strict orthogonality of the HHD, we calculate as follows:
\begin{eqnarray*}
0 &=& {\bf u} ({\bf x}) \cdot \nabla V ({\bf x})\\
&=&\left( \tilde{\bf u}({\bf x}) + H {\bf x}\right)\cdot\left(\nabla\tilde V({\bf x})+ P {\bf x}\right)\\
&=&\left( \nabla\tilde V({\bf x})+ H {\bf x}\right)\cdot\left(\nabla\tilde V({\bf x})+ P {\bf x}\right)\\
&=& || \nabla\tilde V({\bf x})||^2 + \left(P {\bf x} + H {\bf x}\right)\cdot \nabla\tilde V({\bf x}).
\end{eqnarray*}
In the last line, we used $ PH + H^T P= O$ to eliminate $\left(P {\bf x}\right)\cdot\left(H {\bf x}\right).$ By introducing $B := (P+H)/2,$ we obtain
$$||\nabla\tilde V({\bf x}) + B {\bf x} || = ||B {\bf x}||.$$
Therefore we have
\[
|| \nabla\tilde V({\bf x}) || \leq ||\nabla\tilde V({\bf x}) + B {\bf x} || + ||B {\bf x}|| = 2 ||B {\bf x}|| \leq C ||{\bf x}||,
\]
for some positive constant $C.$

From this inequality, we obtain the desired estimate as follows:
\begin{align*}
\left|\tilde V({\bf x}) - \tilde V({\bf 0})\right| &= \left|\int_{0}^{1} \frac{d}{ds} \tilde V(s {\bf x}) ds \right|\\
&= \left|\int_{0}^{1}{\bf x}\cdot \nabla \tilde V(s {\bf x}) ds \right|\\
&\leq \int_{0}^{1}||{\bf x}||\cdot|| \nabla \tilde V(s {\bf x}) ||ds\\
&\leq \frac{C}{2} || {\bf x}||^2.
\end{align*}

By Lemma \ref{lem_p_g}, $\tilde V$ is a polynomial of degree 2 at he most. Here, the gradient and the Hessian matrix of $\tilde V$ vanish at $\textbf{0}.$ Therefore, $\tilde V$ is constant. This shows that $\nabla V({\bf x}) = P {\bf x}$ and ${\bf u}({\bf x}) = H {\bf x}.$

By substituting $H = A +P$ to $ PH + H^T P= O,$ we obtain equation (\ref{lyapunov_eqn}).

The converse follows immediately from Corollary \ref{cor_so}.
\end{proof}

Equation (\ref{lyapunov_eqn}) is a special case of the algebraic Riccati equation, and it is not straightforward to solve exactly although numerical methods have been proposed. In particular, the solutions are not necessarily unique. Details can be found in \cite{riccati}.

If $A$ is normal, there is an obvious solution to equation (\ref{lyapunov_eqn}), which is given by
$$ P = -\frac{1}{2} (A + A^T).$$
Thus we obtain the following theorem:
\begin{theorem}[Theorem \ref{thm_norm}]
Let ${\bf F}({\bf x})= A {\bf x}$ be a vector field on $\mathbb{R}^n$, where $A$ is a normal $n\times n$ rectangular matrix.  Then there is a strictly orthogonal HHD of ${\bf F}$.
\end{theorem}
\begin{remark}
\normalfont
The condition on the normality is necessary for the choice of decomposition $P = -\frac{1}{2} (A + A^T)$ to be applicable.
In fact, by substituting $P = -\frac{1}{2} (A + A^T)$ to equation (\ref{lyapunov_eqn}), we observe that the decomposition given by $ P = -\frac{1}{2} (A + A^T)$ is strictly orthogonal if and only if $A$ is normal.
\end{remark}
\begin{example}
\normalfont
Let ${\bf F}({\bf x})= A {\bf x}$ be a vector field, where $A$ is given by
$$\left(
\begin{array}{ccc}
3 & 0 & -4\\
0 & -1 & 0\\
4 & 0 & 3
\end{array}
\right),$$
which is normal.
Then, a strictly orthogonal HHD may be obtained by setting
$$ P = -\frac{1}{2} (A + A^T) = \left(
\begin{array}{ccc}
-3 & 0 & 0\\
0 & 1 & 0\\
0 & 0 & -3
\end{array}
\right).$$
It is noteworthy that the origin is a saddle in this example. Therefore $V({\bf x}) = {\bf x}^T P {\bf x}/ 2$ cannot be a Lyapunov function, although it decreases along the solutions.
\end{example}
This type of decomposition is considered in \cite{simple_dec}, although the treatment is different from that described herein. Because the condition of normality is highly restrictive, an explicit formula would be useful for more general cases. This is feasible, at the least, for two-dimensional systems.
\begin{proposition}\label{2d_lin_ex}
Let ${\bf F}({\bf x})= A {\bf x}$ be a vector field, where \[
A = \left(
\begin{array}{cc}
a & b \\
c & d
\end{array}
\right)\in {\rm Mat}(2, \mathbb{R}).
\] Then, there exists a strictly orthogonal HHD for ${\bf F}$. Furthermore, if $a+d$ and $b-c$ are non-zero, it is given by the following formula:

\[
P = -\left(
\begin{array}{cc}
\alpha & \beta \\
\beta & a + d - \alpha
\end{array}
\right),
\]
where
\[\begin{split}
\alpha &= \frac{a+d}{(a+d)^2+(b-c)^2} \left( a(a+d)-c(b-c)\right),\\
\beta &= \frac{a+d}{(a+d)^2+(b-c)^2} \left( c(a+d)+a(b-c)\right).
\end{split}\]
\end{proposition}
\begin{proof}
If $b-c = 0,$ we may take $P = -A$ and $H=O.$ Similarly if $a+d =0,$ we may take $P=O$ and $H =A.$

Let us assume that $a+d$ and $b-c$ are non-zero. Let us consider
\[
G =-P = \left(
\begin{array}{cc}
\alpha & \beta \\
\beta & a + d - \alpha
\end{array}
\right),
\]
where $\alpha$ and $\beta$ are unknown.
First, we calculate
\begin{eqnarray*}
&&G( A - G) =\\
&&\left(
\begin{array}{cc}
a\alpha+c \beta -\alpha^2 -\beta^2 & b \alpha - a\beta \\
a\beta-c\alpha -(a+d) \beta +c(a+d) & -\alpha^2 -\beta^2 +(2a+d) \alpha+b \beta -a(a+d)
\end{array}
\right).
\end{eqnarray*}
Because the condition of strict orthogonality (\ref{lya_eqn}) is equivalent to
\[ G(A-G) + \left(G (A-G)\right)^T = O,\]
we obtain the equations
\begin{align}
a\alpha+c \beta -\alpha^2 -\beta^2 &= 0, \label{eqn1} \\
-\alpha^2 -\beta^2 +(2a+d) \alpha+b \beta -a(a+d) &= 0, \label{eqn2} \\
(b-c)\alpha -(a+d)\beta + c(a+d) &=0, \label{eqn3}
\end{align}
by calculating each entry of the matrix in the left-hand side.
Subtracting (\ref{eqn1}) from (\ref{eqn2}), we obtain
\begin{eqnarray*}
(b-c)\alpha -(a+d)\beta + c(a+d) &=&0, \\
(a+d)\alpha + (b-c)\beta - a(a+d) &=&0.
\end{eqnarray*}
The solution for this system of equations is given by
\begin{eqnarray*}
\alpha &=& \frac{a+d}{(a+d)^2+(b-c)^2} \left( a(a+d)-c(b-c)\right),\\
\beta &=& \frac{a+d}{(a+d)^2+(b-c)^2} \left( c(a+d)+a(b-c)\right).
\end{eqnarray*}
It can be verified that these satisfy the condition (\ref{eqn3}).
\end{proof}
From the construction of the decomposition and Theorem \ref{thm_ch}, we obtain the following result.
\begin{corollary}
If the matrix given by
\[
A = \left(
\begin{array}{cc}
a & b \\
c & d
\end{array}
\right)\in {\rm Mat}(2, \mathbb{R}).
\]
satisfies the conditions $a + d \neq 0$ and $b -c \neq 0,$ then the strictly orthogonal HHD of ${\bf F}({\bf x})= A {\bf x}$ exists uniquely.
\end{corollary}

\begin{example}\label{ex_vdp}
\normalfont
Let ${\bf F}({\bf x})= A_{\mu} {\bf x}$ be a vector field, where $A_{\mu}$ is given by
$$\left(
\begin{array}{cc}
0 & 1 \\
-1 & \mu\\
\end{array}
\right),$$
where $\mu$ is a parameter. This is obtained by considering the linearization matrix of the Van der Pol oscillator in a two-dimensional form at the origin.

For this matrix, $\alpha$ and $\beta$ are as follows:
\begin{eqnarray}
\alpha &=& \frac{2 \mu}{\mu^2+4},\\
\beta &=& -\frac{\mu^2}{\mu^2+4}.
\end{eqnarray}
Therefore we have
$$P= -\frac{\mu}{\mu^2+4}\left(
\begin{array}{cc}
2 & -\mu \\
-\mu & \mu^2 + 2
\end{array}
\right).$$
Subsequently, we can obtain
$$H= A_{\mu} +P =\frac{1}{\mu^2+4}\left(
\begin{array}{cc}
-2\mu & 2 \mu^2 + 4\\
- 4 & 2\mu
\end{array}
\right).$$

\end{example}

The construction given above is for linear systems; however it can be applied to analyze general nonlinear vector fields. We remark that it is necessary for a general nonlinear vector field to have a strictly orthogonal HHD that its linear part has one. This follows from the next proposition, whose proof is similar to Lemma \ref{HHD_mat}:
\begin{proposition}
Let ${\bf F}$ be a vector field, not necessarily linear, and ${\bf F} =- \nabla V+ {\bf u} $ be a strictly orthogonal HHD. If ${\bf F}({\bf 0}) = {\bf 0}$, then $ {\bf J}_{\textbf{F}}({\bf 0}) = -\textbf{J}_{\nabla V}({\bf 0}) + \textbf{J}_{\bf u}({\bf 0})$ gives a strictly orthogonal HHD of matrices.
\end{proposition}
Therefore, provided that the linear approximation is valid, we may obtain a Lyapunov function of the origin based on the construction of the HHD presented above. Next, we illustrate this by an analysis of the Van der Pol oscillator.
\begin{example}
\normalfont
Let us consider the system of ordinary equations of the Van der Pol oscillator:
\begin{equation}\label{eqn_vdp}
\begin{split}
\frac{d x}{d t} &= y,\\
\frac{d y}{d t} &= \mu (1- x^2)y -x,
\end{split}
\end{equation}
where $\mu \in \mathbb{R}$ is a parameter. Figure \ref{vdp_pp} shows the phase portrait of the system of equations (\ref{eqn_vdp}).

\begin{figure}[htp]
\begin{center}
\includegraphics[width=8cm]{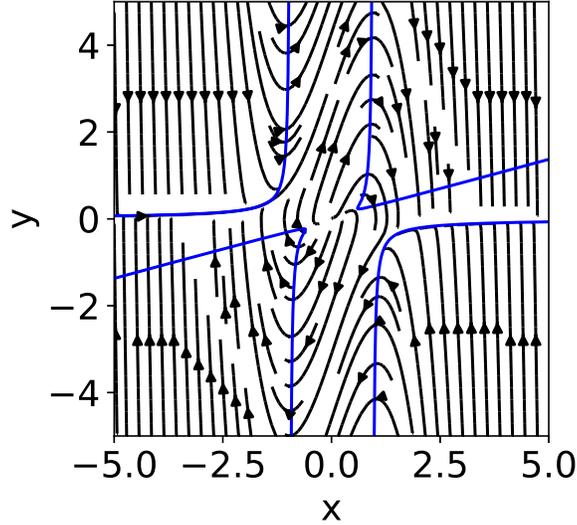}
\caption{Phase portrait of the system of equations (\ref{eqn_vdp}) for $\mu = 3.$ The set $\{ {\bf x} \mid \dot W({\bf x})=0\}$ is plotted for comparison with Figure \ref{vdp_level}.} \label{vdp_pp}
\end{center}
\end{figure}

Let us consider the linear approximation around the origin and apply the HHD to the linear part of the vector field.

First we apply the HHD to the linear part. This is considered in in Example \ref{ex_vdp}, and we obtain a function $W({\bf x}) := \frac{1}{2}{\bf x}^T P {\bf x},$ where $P$ is the matrix constructed therein. For $\mu <0,$ the origin is a sink equilibrium point. In this case, $W$ is a Lyapunov function of the origin.

Next, let us consider the case $\mu >0,$ where the origin is a source equilibrium point. 
Because $P$ is negative definite, the set $\{ {\bf x} \mid W({\bf x})>\lambda\}$ is a filled ellipse if $\lambda < 0$. By considering the linear part, we observe that $\dot W < 0$ holds in a neighborhood of the origin.

Although the stability of the equilibrium is a prerequisite for the definition of the Lyapunov function, a similar analysis can be performed to estimate the basin of repulsion. As long as the condition $\{ {\bf x} \mid W({\bf x})>\lambda\} \subset \{ {\bf x} \mid \dot W({\bf x})<0\}$ holds, $\{ {\bf x} \mid W({\bf x})>\lambda\}$ will be contained in the basin of repulsion. This is established by changing the time parameter to $-t$. With this change, the origin is a sink equilibrium point and $W$ is a Lyapunov function. Figure \ref{vdp_level} shows some level sets of $W$ and the set $\{ {\bf x} \mid \dot W({\bf x})=0\}.$

\begin{figure}[htp]
\begin{center}
\includegraphics[width=8cm]{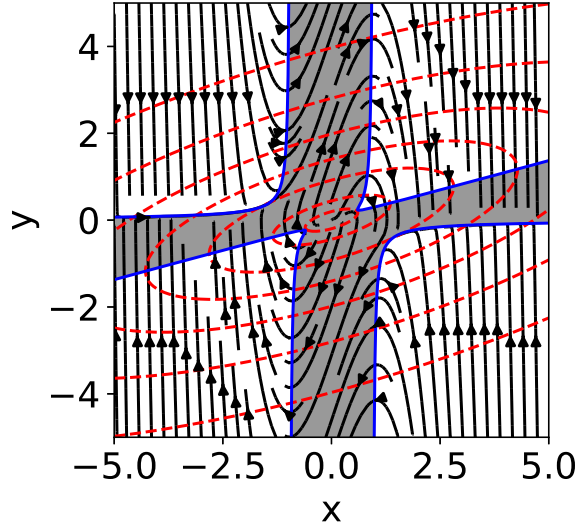}
\caption{Level sets of $W$ and the set $\{ {\bf x} \mid \dot W({\bf x})=0\}$ for $\mu = 3,$  denoted in dashed red lines and blue line, respectively. In the shaded domain, $\dot W$ is positive.} \label{vdp_level}
\end{center}
\end{figure}

Let us obtain further information regarding the behavior of the system by examining the function $W.$
We observe in the phase portrait that there is a ``singular'' orbit $y = y_s(x)$ $(x >1)$ near the level set $\dot W = 0$ (Figure \ref{vdp_pp}), where the orbits from above or below are acutely bent. Below we establish that $y_s$ is estimated by
\[
\frac{x}{\mu(1-x^2)} < y_s(x) < \frac{-2 \mu x + \mu x^3 - x \sqrt{-4 + 8 x^2 + \mu^2 x^4}}{2 \left( (2+ \mu^2) x^2 -(\mu^2+1)\right)},
\]
for $x >1.$

Let us consider the domain $D$ defined by
\[
 D = \{(x,y) \mid x > c_0\text{ and } y < \gamma(x)\}
\]
where 
\[
\begin{split}
c_0 &:=  \sqrt{\frac{1+\mu^2}{2+\mu^2}},\\
\gamma(x) &:= \frac{-2 \mu x + \mu x^3 - x \sqrt{-4 + 8 x^2 + \mu^2 x^4}}{2 \left( (2+ \mu^2) x^2 -(\mu^2+1)\right)}.
\end{split}
\]
Here the curve of the boundary $y = \gamma(x)$ satisfies $\dot W(x, \gamma(x))=0.$ This condition implies that the vector field ${\bf F}({\bf x})$ is tangent to the level set of $W$ at each ${\bf x} \in \partial D.$ Further, if the vector ${\bf F}({\bf x})$ points inward (outward) at a point ${\bf x} \in \partial D$, the same holds for a neighborhood of ${\bf x}$ provided that the level sets of $W$ are transverse to $\partial D.$ Therefore we can determine the direction of each orbit passing $\partial D.$ 

 Additionally, $\lambda_0 < 0$ exists such that $W = \lambda_0$ is tangent to $\partial D$ at exactly one point ${\bf x}_c \in \partial D.$ $\partial D$ is separated into two parts by ${\bf x}_c$. By studying the direction of the vector field, we conclude that the orbits starting within this set cannot leave it from the upper part of $\partial D$.

We further restrict the possible motion of the orbits by considering the $y$-nullcline given by
$$ y = \frac{x}{\mu(1-x^2)}.$$
Because it is monotonous for $x > 1,$ the set 
\[
\{(x,y) \mid x > 1 \text{ and }y < \frac{x}{\mu(1-x^2)}\}
\] is backward invariant.

It can be shown that the curve $y = \gamma(x)$ lies above the $y$-nullcline if $\mu > 0$ and $x > 1.$ Therefore, the orbits entering $D$ from the upper edge remain in the narrow strip between these two curves. Figure \ref{vdp_level_detail} shows the curve $y = \gamma(x)$ and the $y$-nullcline.

\begin{figure}[htp]
\begin{center}
\includegraphics[width=8cm]{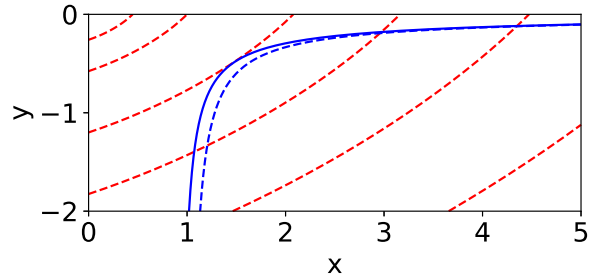}
\caption{Level sets of $W$, the curve $y = \gamma(x)$ and the $y$-nullcline for $\mu = 3,$ denoted in dashed red lines, blue line and dashed blue line, respectively. } \label{vdp_level_detail}
\end{center}
\end{figure}

Hence we observe that, between the curve $y = \gamma(x)$ and the $y$-nullcline, the direction of the vector field changes acutely (Figure \ref{vdp_pp_detail}). Therefore, the singular orbit is contained in this region, and we have established the desired estimate.
\begin{figure}[htbp]
\begin{center}
\includegraphics[width=8cm]{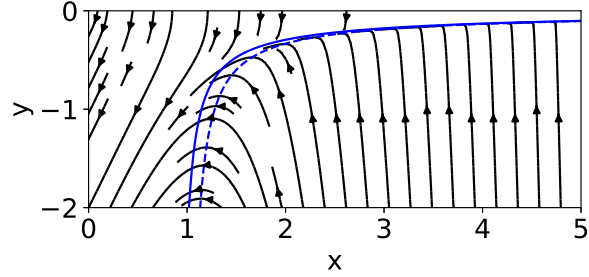}
\caption{Phase portrait of the system of equations (\ref{eqn_vdp}) around the curve $y = \gamma(x)$ and the $y$-nullcline.}\label{vdp_pp_detail}
\end{center}
\end{figure}

\begin{figure}[htbp]\label{vdp_pp_detail_pc}
\begin{center}
\includegraphics[width=8cm]{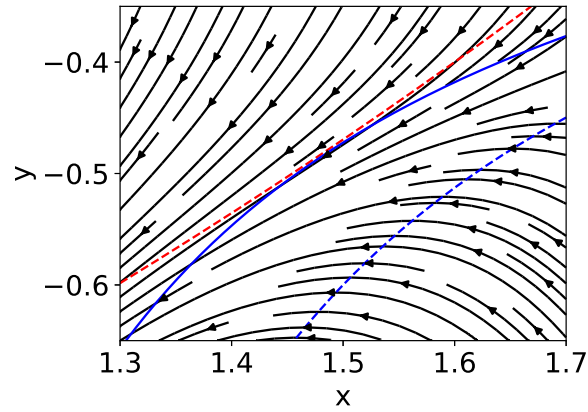}
\caption{Phase portrait of the system of equations (\ref{eqn_vdp}) around the point ${\bf x}_c$. Level sets of $W$, the curve $y = \gamma(x)$, and the $y$-nullcline are plotted.}
\end{center}
\end{figure}

\end{example}

At this point, let us discuss the relationship between HHD and SDE decomposition, which provides us with a method to calculate Lyapunov functions.

The SDE decomposition is a decomposition method of vector fields first introduced to construct potential functions for stochastic differential equations \cite{ao2004potential, kwon2005structure, yuan2017sde}. By setting the noise term to zero, it can be applied to ODE defined by vector fields. For a linear vector field ${\bf F}({\bf x}) = F {\bf x}$, it amounts to consider, for a prescribed symmetric matrix $D$, a decomposition of the form
\begin{equation}\label{SDED}
F = -(D + Q) U,
\end{equation}
where Q is skew-symmetric and U is symmetric \cite{kwon2005structure}.
Because $\rm{tr} (QU) = 0$ holds, SDE decomposition (\ref{SDED}) yields an HHD $F=-P+H$ by setting

\begin{equation}\label{SDED_HHD}
\begin{aligned}
P &= DU,\\
H &= -QU,
\end{aligned}
\end{equation}
if and only if $DU$ is symmetric, that is, $D$ commutes with $U$. In particular, a strictly orthogonal HHD is obtained if $D=I$.

Conversely, a strictly orthogonal HHD $F = -P +H$ with regular gradient part $P$ can be expressed as an SDE decomposition via (\ref{SDED_HHD}). Indeed, the matrices $D = I,$ $U=P$ and $Q = -H P^{-1}$ satisfy the required conditions. Thus, in the case of linear vector fields, SDE decomposition with $D=I$ and strictly orthogonal HHD are equivalent in the sense that they correspond via (\ref{SDED_HHD}).

Let us note that this equivalence is not trivial because we can construct an example of HHD without corresponding SDE decompositions. Therefore, the equivalence with SDE decomposition may be regarded as evidence that the strict orthogonality is not an artificial choice.
The next proposition states a necessary and sufficient condition for these two methods to yield an equivalent result.
\begin{proposition}
Let $D$ be a symmetric regular matrix, and $F = -P + H$ an HHD of a matrix.
There exists a pair of a skew-symmetric matrix $Q$ and a symmetric matrix $U$ satisfying (\ref{SDED_HHD}) if and only if $D$ commutes with $P$ and the following condition is satisfied for a skew-symmetric matrix $Q$:
\begin{equation}\label{SDED_HHD_equiv}
QP + HD = O.
\end{equation}
\end{proposition}
\begin{proof}
Let $Q$ be a skew-symmetric matrix and $U$ a symmetric matrix satisfying (\ref{SDED_HHD}). Then, $U$ commutes with $D$ because $P$ is symmetric. Therefore, $D$ commutes with $P$. The condition (\ref{SDED_HHD_equiv}) is verified by direct calculations.

For the converse, set $U := D^{-1} P.$
\end{proof}

\begin{example}
\normalfont
Let us consider the next HHD of a matrix:
\[
\left( 
\begin{array}{ccc}
	-1 & 1 & 1 \\
	2 & -1 & 1	\\
	2 & 2 & -1
\end{array}
 \right)
 = -I+
\left( 
\begin{array}{ccc}
	0 & 1 & 1 \\
	2 & 0 & 1	\\
	2 & 2 & 0
\end{array}
 \right).
\]
Then, no skew-symmetric nonzero matrix satisfies the equation (\ref{SDED_HHD_equiv}) for any choice of $D$. Therefore, no SDE decomposition satisfies (\ref{SDED_HHD}).
\end{example}

\section{Analysis of planar systems}\label{a_pla_sys}
In this section, we consider the analysis of planar systems. For planar systems, the HHD assumes a simpler form that can be regarded as a generalization of the complex potential, which is an important tool in hydrodynamics. This complex potential formalism of the HHD enables us to construct the decomposition exactly if a vector field is given by analytic functions.

First, we observe that, in the planar case, a vector field ${\bf u}$ is integrable if $\nabla \cdot {\bf u} $ identically vanishes. That is, if we set
$$ J := \left( \begin{array}{cc}
0 & 1 \\
-1 & 0
\end{array} \right),$$
we have ${\bf u} = J \nabla H$ for some differentiable function $H$.

This motivates us to establish the following definition.

\begin{definition}
\normalfont
For a smooth vector field $ {\bf F}$ on $\mathbb{R}^2$, its HHD takes the following form:
\begin{equation}\label{ph_dec}
{\bf F}({\bf x}) = -\nabla V({\bf x}) + J\nabla H({\bf x}),
\end{equation}
where $V({\bf x})$ and $ H({\bf x})$ are scalar functions. We will call $V$ a {\bf potential function} and $H$ a {\bf Hamiltonian function}.
\end{definition}
\begin{remark}
\normalfont
In Demongeot, Grade and Forest, this type of HHD is called the potential-Hamiltonian decomposition \cite{Demongeot_Li_2007}.
\end{remark}

A combination of two functions appearing (\ref{ph_dec}) gives us a generalization of the complex potential, which is a useful tool in the analysis of two-dimensional flows in hydrodynamics.
\begin{definition}[Complex potential]
\normalfont
For a HHD given by ${\bf F}({\bf x}) = -\nabla V({\bf x}) + J\nabla H({\bf x})$, its corresponding {\bf complex potential} $W$ is defined by
\begin{equation}\label{comp}
W = 2 (-V + iH).
\end{equation}
If $W$ is a complex-valued function obtained by the procedure above, we say that $W$ is a complex potential of the planar vector field ${\bf F}.$
\end{definition}
Therefore, we may use complex potentials instead of the two functions appearing in the HHD.
In what follows, we use the next notation:
\begin{definition}[Wirtinger derivatives]
\normalfont
For a real or complex valued smooth function $\phi$ defined on $\mathbb{R}^2,$ we define
\begin{eqnarray*}
\frac{\partial \phi}{\partial z} &:=& \frac{1}{2} \left( \frac{\partial \phi}{\partial x} - i \frac{\partial \phi}{\partial y}\right),\\
\frac{\partial \phi}{\partial \bar z} &:=& \frac{1}{2} \left( \frac{\partial \phi}{\partial x} + i \frac{\partial \phi}{\partial y}\right).
\end{eqnarray*}
\end{definition}
It is easy to observe that the complex conjugate of $\frac{\partial \phi}{\partial z}$ is $\frac{\partial \bar\phi}{\partial \bar z}.$ Furthermore, it can be verified that the chain rule holds. Details on the Wirtinger calculus may be found in \cite{burckel2012theory}.

In the complex potential formalism, the system of ordinary equations is written in a simpler form.
\begin{proposition}
Let us assume that a system of ordinary equations is given by
\begin{eqnarray*}
\frac{d x}{d t} &=& f(x,y),\\
\frac{d y}{d t} &=& g(x,y),
\end{eqnarray*}
where $f$ and $g$ are smooth functions. If $W$ is a complex potential for a HHD of the vector field ${\bf F} (x,y) = (f(x,y), g(x,y))^T,$ then we have
\begin{equation}\label{comp_eq}
\frac{d \bar z}{ d t} = \frac{\partial W}{\partial z},
\end{equation}
where $z = x + iy.$
\end{proposition}
\begin{proof}
This is verified by a direct calculation as follows:
\begin{eqnarray*}
\frac{\partial W}{\partial z} &=& \left( \frac{\partial}{\partial x} - i \frac{\partial}{\partial y}\right) (-V + iH)\\
&=& \left( -\frac{\partial V}{\partial x} + \frac{\partial H}{\partial y}\right) +i \left( \frac{\partial H}{\partial x} + \frac{\partial V}{\partial y}\right)\\
&=& f(x,y) - i g(x,y)\\
&=& \frac{d \bar z}{ d t}.
\end{eqnarray*}
\end{proof}
The next corollary follows from the property of the Wirtinger derivatives. It characterizes all the  possible decompositions.
\begin{corollary}
Let $W$ be a complex potential for a planar vector field ${\bf F}$ and $\phi(z)$ be a holomorphic function. Then, $W +\phi(\bar z)$ is also a complex potential for the vector field ${\bf F}$.

Conversely, let $W$ and $\tilde W$ be two complex potentials for the same vector field. Then $W - \tilde W = \phi(\bar z)$ for some holomorphic function $\phi(z).$ That is, the difference of two complex potentials is an antiholomorphic function.
\end{corollary}
The notion of orbital derivatives is important in the theory of Lyapunov functions. For a complex potential, it assumes a simple form.
\begin{corollary}
If $W$ is a complex potential, then the following identity holds.
$$\dot W = \left|\frac{\partial W}{\partial z} \right|^2 + \frac{\partial W}{\partial \bar z}\frac{\partial W}{\partial z}$$
In particular, if $W$ is holomorphic, then we have
$$\dot W = \left|\frac{\partial W}{\partial z} \right|^2.$$
\end{corollary}
\begin{proof}
This is verified by a direct calculation:
\begin{eqnarray*}
\dot W &=& \frac{\partial W}{\partial x}\frac{d x}{ d t} + \frac{\partial W}{\partial y}\frac{d y}{ d t}\\
&=& \frac{1}{2} \left\{ \left(\frac{\partial W}{\partial x} +i \frac{\partial W}{\partial y} \right) \left(\frac{d x}{ d t} -i \frac{d y}{ d t} \right)+ \left(\frac{\partial W}{\partial x} -i \frac{\partial W}{\partial y} \right) \left(\frac{d x}{ d t} +i \frac{d y}{ d t} \right)\right\}\\
&=& \frac{\partial W}{\partial \bar z}\frac{d \bar z}{ d t} +\frac{\partial W}{\partial z}\frac{d z}{ d t}\\
&=&\frac{\partial W}{\partial \bar z}\frac{\partial W}{\partial z} + \left|\frac{\partial W}{\partial z} \right|^2.
\end{eqnarray*}

\end{proof}
The preceding corollary implies that the vector fields with holomorphic or antiholomorphic complex potentials are easy to analyze because the motion of orbits can be tracked by considering the level sets of $V$ or $H.$ The antiholomorphic case is trivial because it corresponds to the case where ${\bf F}$ is identically zero. The next proposition gives us a characterization of holomorphic cases. Here we denote $\nabla \cdot (J\,{\bf F})$ by $\nabla \times {\bf F}.$

\begin{proposition}
Let ${\bf F}$ be a smooth vector field defined on $\mathbb{R}^2.$ Then we have
$\nabla \cdot {\bf F} = \nabla \times {\bf F} = 0$ if and only if there exists a holomorphic complex potential of ${\bf F}$.
\end{proposition}
\begin{proof}
If ${\bf F}$ is a smooth vector field with $\nabla \cdot {\bf F} = \nabla \times {\bf F} = 0,$ it is well-known that a holomorphic complex potential can be obtained explicitly \cite{Chorin_A_1990}.

Let us assume that a holomorphic complex potential $W$ of ${\bf F}$ exists. From the Cauchy-Riemann equations, we have
$${\bf F} = -2 \nabla V,$$
where $V$ is harmonic. Therefore, $\nabla \cdot {\bf F} = \nabla \times {\bf F} = 0$ holds.
\end{proof}
Equation (\ref{comp_eq}) presents a method to calculate a complex potential explicitly. If we integrate the left-hand side of the equation formally with respect to $z$, we may obtain all possible decompositions for a vector field.
\begin{example}
\normalfont
Let us consider a vector field given by
$$
{\bf F}(x,y) := \left(\begin{array}{c}
-x^2+2y\\
-y^2+x
\end{array}
\right).
$$
We calculate its HHD by the following procedures. First we define a function by
$$F(x,y) := (-x^2+2y) - i(-y^2+x).$$
This corresponds to $\frac{d \bar z}{d t}.$ Here we change the variables of $F(x,y)$ by
\begin{eqnarray*}
x &=& \frac{1}{2} (z + \bar z)\\
y &=& \frac{1}{2 i} (z - \bar z).
\end{eqnarray*}
This gives us
$$F(z,\bar z) = -\frac{1+i}{4} z^2 -\frac{1-i}{2} |z|^2 -\frac{1+i}{4}\bar z^2 -\frac{i}{2}(3 z - \bar z).$$
Next, we integrate $F(z, \bar z)$ with respect to $z$. This yields a complex potential as follows:
$$W(z,\bar z) = -\frac{1+i}{12} z^3 -\frac{1-i}{4} |z|^2z -\frac{1+i}{4}|z|^2\bar z -\frac{i}{2}(\frac{3}{2} z^2 - |z|^2).$$
By considering its real or imaginary part, we obtain a HHD of ${\bf F}:$
\begin{eqnarray*}
V &=& \frac{7 x^3}{24} + \frac{x^2 y}{8} + \frac{x y^2}{8} - \frac{3 x y}{4} + \frac{7 y^3}{24}\\
H &=& -\frac{x^3}{24} - \frac{x^2 y}{8} - \frac{x^2}{8} + \frac{x y^2}{8} + \frac{y^3}{24} + \frac{5 y^2}{8}.
\end{eqnarray*}
\end{example}
Although this calculation method requires the vector field to be given in analytic functions, it can be applied to more general cases if we fit the data with, for example, polynomial vector fields.

 The strict orthogonality of the HHD is characterized as follows:
\begin{lemma}
Let there be an HHD given by ${\bf F}({\bf x}) = -\nabla V({\bf x}) + J\nabla H({\bf x})$ and $W$ be the corresponding complex potential. Then, the HHD is strictly orthogonal if and only if the following identity holds.
$$\left|\frac{\partial W}{\partial z} \right|^2 = \left|\frac{\partial W}{\partial \bar z} \right|^2.$$
\end{lemma}
\begin{proof}
Let us consider the next form:
$$4\frac{\partial V}{\partial \bar z}\frac{\partial H}{\partial z}.$$
By a direct calculation, we can demonstrate that this is equivalent to
$$\frac{\partial V}{\partial x}\frac{\partial H}{\partial x}+\frac{\partial V}{\partial y}\frac{\partial H}{\partial y} + i \left( -\frac{\partial V}{\partial x}\frac{\partial H}{\partial y}+\frac{\partial V}{\partial y}\frac{\partial H}{\partial x} \right).$$
On the other hand, we have
\begin{eqnarray*}
V &=& -\frac{1}{4}(W + \bar W)\\
H &=& \frac{1}{4 i}(W - \bar W).
\end{eqnarray*}
Therefore we have
$$4\frac{\partial V}{\partial \bar z}\frac{\partial H}{\partial z} = \frac{i}{4}\left( \frac{\partial W}{\partial z} \frac{\partial W}{\partial \bar z} -\frac{\partial W}{\partial z} \frac{\partial \bar W}{\partial \bar z}+\frac{\partial \bar W}{\partial z} \frac{\partial W}{\partial \bar z}-\frac{\partial \bar W}{\partial z} \frac{\partial \bar W}{\partial \bar z}\right)$$
Noting that the complex conjugate of $\frac{\partial \bar W}{\partial z} \frac{\partial \bar W}{\partial \bar z}$ is $\frac{\partial W}{\partial z} \frac{\partial W}{\partial \bar z},$ we obtain
$$ -\frac{\partial V}{\partial x}\frac{\partial H}{\partial y}+\frac{\partial V}{\partial y}\frac{\partial H}{\partial x} = \frac{1}{4} \left( \left|\frac{\partial W}{\partial \bar z} \right|^2 - \left|\frac{\partial W}{\partial z} \right|^2\right)$$
by comparing the imaginary part.
\end{proof}
This lemma gives us a method to construct a strictly orthogonal HHD. Namely, we may solve the following equation for $\phi$ by the power series method:
\begin{equation}\label{ohhd_eqn}
\left|\frac{\partial W_0}{\partial z} \right|^2 = \left|\frac{\partial W_0}{\partial \bar z} + \phi'(\bar z) \right|^2,
\end{equation}
where $W_0$ is a fixed complex potential and $\phi$ is a holomorphic function.
\begin{example}
\normalfont
Here we present another proof of Proposition \ref{2d_lin_ex} by solving equation (\ref{ohhd_eqn}). For linear vector fields on $\mathbb{R}^2,$ the function $F = f-ig$ assumes the from $F= a z + b \bar z,$ where $a$ and $b$ are complex constants. Then a complex potential 
$$W_0= \frac{a}{2}z^2+b |z|^2$$
is obtained by integrating it with respect to $z$. Therefore equation (\ref{ohhd_eqn}) is now in the following form:
\begin{equation}\label{ohhd_eqn_lin}
\left|a z + b \bar z\right|^2 = \left|b z+ \phi'(\bar z) \right|^2.
\end{equation}
If $b$ = 0, then equation (\ref{ohhd_eqn_lin}) has a solution 
$$\phi(z) = \frac{a}{2} z^2.$$
 If $b \neq 0,$ the function
$$\phi(z) = \frac{\bar a b}{2 \bar b} z^2$$
is verified to be a solution. This establishes the existence of a strictly orthogonal HHD.
\end{example}
The next theorem is obtained similarly.
\begin{theorem}[Theorem \ref{thm_quad}]
Let there be given a system of ordinary equations of the following form:
\begin{eqnarray*}
\frac{d x}{d t} &=& f(x,y) = p_1 x^2 + q_1 xy + r_1 y^2,\\
\frac{d y}{d t} &=& g(x,y) = p_2 x^2 + q_2 xy + r_2 y^2,
\end{eqnarray*}
where $p_i, q_i, r_i$ $(i=1,2)$ are real constants. If the coefficients satisfy the condition
\begin{equation}\label{cond_ohhd_e}
q_1^2 -2 (p_2-r_2) q_1 -4 p_2 r_2 + q_2^2 +2 (p_1-r_1) q_2 -4 p_1 r_1 = 0,
\end{equation}
then there exists a strictly orthogonal HHD of the vector field ${\bf F}(x,y) = \left(f(x,y),g(x,y)\right)^T$.
\end{theorem}
\begin{proof}
First, we obtain the coefficients $a, b, c \in \mathbb{C}$ determined by the identity 
$$ a z^2 + b |z|^2 + c\bar z^2 = p_1 x^2 + q_1 xy + r_1 y^2 -i \left(p_2 x^2 + q_2 xy + r_2 y^2\right).$$
By a direct calculation, we have
\begin{eqnarray*}
	a &=& \frac{1}{4}\left( (p_1-q_2-r_1) + i(-p_2-q_1+r_2) \right)\\
	b &=& \frac{1}{2}\left( (p_1+r_1) - i(p_2+r_2) \right)\\
	c &=& \frac{1}{4}\left( (p_1+q_2-r_1) + i(-p_2+q_1+r_2) \right).\\
\end{eqnarray*}
Next, we find a complex potential $W_0(z, \bar z)$ by integrating $ a z^2 + b |z|^2 + c\bar z^2$ with respect to $z$:
$$W_0(z, \bar z) = \frac{a}{3} z^3 + \frac{b}{2} |z|^2 z + c |z|^2 \bar z.$$
Then, equation (\ref{ohhd_eqn}) assumes the following form:
$$\left|a z^2 + b |z|^2 + c\bar z^2\right|^2 = \left|\frac{b}{2} z^2 + 2 c |z|^2 + \phi'(\bar z) \right|^2.$$
Here we search for a solution of this equation with the form $\phi(z) = C z^3.$ Let us first consider the case with $b \neq 0.$ By comparing both sides of the equation, we have
$$ C = \frac{2 \bar a c}{3 \bar b},$$
and the necessary and sufficient condition for the existence of a solution is
$$|b|^2 = 4 |c|^2,$$
which is equivalent to equation (\ref{cond_ohhd_e}). For the case with $b = 0,$ it is easy to explicitly construct a strictly orthogonal HHD.
\end{proof}
\begin{remark}
\normalfont
Equation (\ref{cond_ohhd_e}) is satisfied for a divergence-free or rotation-free case.
\end{remark}
\begin{remark}
\normalfont
The complex potential obtained is of the following form:
\begin{equation}\label{s_complex}
W(z, \bar z) = \frac{a}{3} z^3 + \frac{b}{2} |z|^2 z + c |z|^2 \bar z + \frac{2 \bar a c}{3 \bar b} \bar z^3.
\end{equation}
By taking the real part or the imaginary part of $W$, we obtain a strictly orthogonal HHD explicitly. 
\end{remark}
For the systems considered in the last theorem, the method of linearization is not applicable. However, if a strictly orthogonal HHD is obtained, we may determine the stability of the equilibrium points because the potential function may be used as a Lyapunov function.
\begin{example}
\normalfont
Let us consider the following system of equations:
\begin{eqnarray*}
\frac{d x}{d t} &=& x^2 - y^2,\\
\frac{d y}{d t} &=& -x^2 - y^2.
\end{eqnarray*}
For this system, a strictly orthogonal HHD can be obtained immediately:
\begin{eqnarray*}
V &=& -\frac{x^3}{3} + \frac{y^3}{3},\\
H &=& \frac{x^3}{3} - \frac{y^3}{3}.
\end{eqnarray*}
By a direct calculation, it can be verified that the complex potential obtained in equation (\ref{s_complex}) is the same as that presented above. By the form of the potential function $V$, we observe that the origin is not stable.
\end{example}
\begin{example}
\normalfont
Let us consider the following system of equations:
\begin{eqnarray*}
\frac{d x}{d t} &=&  x^2 -2 xy +3y^2,\\
\frac{d y}{d t} &=& 4 x^2 -4 xy +2 y^2,
\end{eqnarray*}
The coefficients satisfy equation (\ref{cond_ohhd_e}); therefore a strictly orthogonal HHD exists. This is given by
$$\frac{1}{6} z^3 + \frac{2-3i}{2} |z|^2 z - \frac{3+2i}{2} |z|^2 \bar z - \frac{1}{6}\frac{3+2i}{2+3i} \bar z^3.$$
Equivalently, in a more useful form:
$$-\frac{19 x^3}{39} + \frac{9 x^2 y}{13} + \frac{17 y^3}{39}- \frac{7 x y^2}{13}+ i \left(-\frac{95 x^3}{39} + \frac{45 x^2 y}{13} - \frac{35 x y^2}{13} + \frac{85 y^3}{39}\right).$$
By a direct calculation, it can be verified that this potential yields a strictly orthogonal HHD.
\end{example}
It is noteworthy that a strictly orthogonal HHD does not necessarily exist, and the method may fail.

\section{Concluding remarks}\label{con_rem}

Vector fields with strictly orthogonal HHDs are as ``easy'' as gradient vector fields in that the details of their behavior can be determined by studying a function. We presented some methods to construct such a decomposition in the case of linear or planar vector fields. Further, we demonstrated its use by examples including the Van der Pol oscillator. Strictly orthogonal HHDs are useful if they are available.

However, it is clear that the if statement in the previous sentence is rather ambitious. A general vector field does not admit a strictly orthogonal HHD, as observable in results such as Theorem \ref{thm_quad}.   
This is the most substantial obstacle in application, but it could be bypassed. As is evident in the example of the Van der Pol oscillator, an approximation by a strictly orthogonal HHD provides significant information regarding the behavior of a system. Thus, if we can approximate general vector fields with them, we can obtain a useful analysis method. This aspect should be explored in further research.

The relation between strictly orthogonal HHD and the SDE decomposition is another topic that deserves further research. In particular, the equivalence between these two methods is of interest because it provides valuable insights. As our discussion on this point is limited to linear cases, consideration of nonlinear systems is desirable.

\section*{Acknowledgements}
I would like to express my gratitude to Professor Masashi Kisaka for his patient guidance and useful critique of this research. This study was supported by Grant-in-Aid for JSPS Fellows (17J03931). I would like to thank Editage (www.editage.jp) for English language editing. Further, I would like to thank the anonymous reviewers for their valuable comments and suggestions.

\bibliographystyle{AIMS}
\bibliography{application_HHD}
\end{document}